\documentclass[reqno]{amsart}
\usepackage{amssymb,amsmath,amsfonts,amscd,amsthm,pb-diagram,amsbsy,bm}
\usepackage{color}
\numberwithin{equation}{section}

\newtheorem{theorem}{Theorem}[section]
\newtheorem{definition}{Definiton}[section]
\newtheorem{main theorem}{Main Theorem}
\newtheorem{lemma}{Lemma}[section]
\newtheorem{proposition}{Proposition}[section]

\theoremstyle{remark}
\newtheorem{remark}{Remark}[section]

\begin{document}

\title[Tchebychev hyperovaloid is ellipsoid]
{Every centroaffine Tchebychev hyperovaloid is ellipsoid}

\author[X. Cheng, Z. Hu and L. Vrancken]
{Xiuxiu Cheng, Zejun Hu and Luc Vrancken}


\thanks {The first author was supported by CPSF, Grant No. 2019M652554.
The second author was supported by NSF of China, Grant No. 11771404.}


\begin{abstract}
In this paper, we study locally strongly convex Tchebychev hypersurfaces,
namely the {\it centroaffine totally umbilical hypersurfaces}, in the
$(n+1)$-dimensional affine space $\mathbb{R}^{n+1}$. We first make an
ordinary-looking observation that such hypersurfaces are characterized by having
a Riemannian structure admitting a canonically defined closed conformal vector
field. Then, by taking the advantage of properties about Riemannian manifolds
with closed conformal vector fields, we show that the ellipsoids are the only
centroaffine Tchebychev hyperovaloids. This solves the longstanding problem of
trying to generalize the classical theorem of Blaschke and Deicke on affine
hyperspheres in equiaffine differential geometry to that in centroaffine
differential geometry.
\end{abstract}

\maketitle

\section{Introduction}\label{sect:1}

In this paper, we study locally strongly convex centroaffine hypersurfaces, i.e., the
hypersurfaces of the $(n+1)$-dimensional affine space $\mathbb{R}^{n+1}$ with centroaffine
normalization. It is well-known that in both Euclidean and equiaffine differential geometry,
the Weingarten (shape) operator contains essential geometric information about a hypersurface.
This is different in centroaffine differential geometry, where one studies the properties
of hypersurfaces in $\mathbb{R}^{n+1}$ which are invariant under the centroaffine
transformation group $G=GL(n+1,\mathbb{R})$, where $G$ keeps the origin of $\mathbb{R}^{n+1}$
fixed. Since the centroaffine normalization induces the identity as Weingarten operator,
from the point of view of relative differential geometry any nondegenerate hypersurface
with centroaffine normalization is a relative hypersphere (see sections 6.3 and 7.2 of
\cite{SSV}); thus in centroaffine differential geometry {\it the usually induced Weingarten
operator} contains no further geometric information.

In such situation, C.P. Wang \cite{W} made a breakthrough by giving the reasonable
definition for the Weingarten (shape) operator on centroaffine hypersurfaces of the
$(n+1)$-dimensional affine space $\mathbb{R}^{n+1}$. Specifically, on a centroaffine
hypersurface, besides the centroaffine metric, there exists a canonically defined
{\it Tchebychev vector field} $T$. Let $\hat\nabla$ denote the Levi-Civita connection with
respect to the centroaffine metric, then the operator $\mathcal{T}:=\hat\nabla T$
was introduced to be defined as {\it the centroaffine shape operator} by Wang \cite{W}.
(Note: eversince \cite{W} the centroaffine shape operator $\mathcal{T}$ is also called the {\it
Tchebychev operator}). To justify this terminology, it was shown that the Tchebychev operator
$\mathcal{T}$ in centroaffine differential geometry is analogous to the shape operator in the
equiaffine differential geometry. Indeed, C.P. Wang \cite{W} calculated the first variation
formula of the centroaffine area functional and, as an important result, he showed (cf. Theorem
2 of \cite{W}) that the critical hypersurfaces of this functional are exactly hypersurfaces with
vanishing {\it centroaffine mean curvature} $H:=\frac1n{\rm Tr}\,\mathcal{T}$; moreover, as in
Euclidean and equiaffine differential geometry, C.P. Wang also proved (Theorem 1 in \cite{W})
that the only hyperovaloid in $\mathbb{R}^{n+1}$ with constant centroaffine mean curvature is
the ellipsoid centered at the origin of $\mathbb{R}^{n+1}$. It is worthy to note that, as there
are no general results about the sign of the second variation of the centroaffine area functional
at the critical hypersurfaces, it was suggested in \cite{LLS} (see also \cite{V}) to call a
centroaffine hypersurface with $H=0$ the {\it centroaffine extremal hypersurface}.

The centroaffine shape operator $\mathcal{T}$ was studied systematically from Liu and Wang
\cite{LW1}. In particular, there is an
important subclass of centroaffine hypersurfaces, namely the {\it centroaffine
totally umbilical hypersurfaces}. By definition, it consists of centroaffine
hypersurfaces whose shape operator $\mathcal{T}$ is proportional to the identity
isomorphism of the tangent spaces. Following Liu and Wang \cite{LW1}, these
centroaffine hypersurfaces are usually referred to as {\it Tchebychev hypersurfaces}.
Obviously, the notion of Tchebychev hypersurfaces generalizes in a natural way the
notion of {\it affine hyperspheres} in equiaffine differential geometry. More to be
pointed out is that both, i.e., Tchebychev hypersurfaces in centroaffine differential
geometry and affine hyperspheres in equiaffine differential geometry, have exactly
the similar structure equations (cf. \cite{HLV,LSZH,LW1,NS,SSV}). Because of such
nice similarity, the Tchebychev hypersurfaces have been under extensive study. For
references, we refer to \cite{B,CHY,LLSSW,L,LSW,LW1,LW2,STV}.

In equiaffine differential geometry, we have the well-known classical theorem of
Blaschke and Deicke (cf. Theorem 3.35 in \cite{LSZH}) which states that
if a {\it hyperovaloid} (which means a connected compact locally strongly
convex hypersurface without boundary in the $(n+1)$-dimensional affine
space $\mathbb{R}^{n+1}$) is an affine hypersphere, then it is an ellipsoid.
The Blaschke and Deicke's theorem and the preceding mentioned similarity between
affine hyperspheres and Tchebychev hypersurfaces motivate strongly to study the
following problem, which will provide an interesting new global characterization
of the ellipsoid as centroaffine Tchebychev hyperovaloid.

\vskip 2mm

\noindent {\bf PROBLEM} (\cite{CHY}). {\it  Let $x: M^n\rightarrow\mathbb{R}^{n+1}$
$(n\ge2)$ be a centroaffine Tchebychev hyperovaloid. Must $x(M^n)$ be an ellipsoid containing
the origin of $\mathbb{R}^{n+1}$}?

\vskip 2mm

The PROBLEM has been considered, first by Liu and Wang \cite{LW1} but restricts to
the case $n=2$. It was solved affirmatively:

\begin{theorem}[cf. Theorem 4.3 of \cite{LW1}]\label{thm:1.1}
Let $x:M^2\rightarrow\mathbb{R}^{3}$ be a centroaffine Tchebychev ovaloid. Then $x(M^2)$ is
an ellipsoid in $\mathbb{R}^{3}$ such that the origin of $\mathbb{R}^{3}$ is in the
inside of $x(M^2)$.
\end{theorem}

The PROBLEM has been further investigated by many researchers in higher
dimensional cases. In \cite{LSW}, Liu, Simon and Wang solved it affirmatively
under an additional {\it nondegenerate equiaffine Gauss map} condition (cf.
Theorem 5.2 of \cite{LSW}). In Theorem 5 of \cite{L}, following an argument
about affine hyperspheres, M. Li also obtained some partial results from the
point of view in relative affine differential geometry. More recently, joint
with Z.K. Yao the first two authors of the present article solved the PROBLEM
affirmatively under the additional condition {\it the centroaffine metric
having nonnegative sectional curvatures} (Theorem 1.7 of \cite{CHY}). We would
mention that the method of \cite{CHY} depends heavily on the recent classification
of locally strongly convex centroaffine hypersurfaces with parallel traceless
difference tensor (cf. \cite{CH,CHM,HLV}).

In this paper, as the continuation of \cite{CHY}, we still focus on the above PROBLEM.
By adopting a new approach, we first make an ordinary-looking but
important observation (Lemma \ref{lem:2.1}) that the Tchebychev vector field
on centroaffine hypersurfaces is closed, and that a centroaffine hypersurface
is a Tchebychev hypersurface if and only if its Tchebychev vector field is a
conformal vector field with respect to the centroaffine metric. Then, by taking
the advantage of typical properties about a Riemannian manifold with closed conformal
vector field, we eventually solve the PROBLEM affirmatively for every dimension
$n\ge3$. Our main result can be stated as follows:

\begin{theorem}\label{thm:1.2}
Let $x: M^n\rightarrow\mathbb{R}^{n+1}$ $(n\ge3)$ be a centroaffine Tchebychev hyperovaloid.
Then $x(M^n)$ is an ellipsoid such that the origin of $\mathbb{R}^{n+1}$ is in the
inside of $x(M^n)$.
\end{theorem}

\begin{remark}\label{rem:1.1}
A centroaffine hypersurface $x: M^n\rightarrow\mathbb{R}^{n+1}$ is a Tchebychev
hypersurface if and only if it satifies $\mathcal{T}=\alpha\,{\rm id}$, where
$\alpha$ is a smooth function on $M^n$. However, different from the affine
hyperspheres in equiaffine differential geometry, where the equiaffine shape
operator must be a constant multiple of the identity isomorphism of the tangent
space, here the function $\alpha$ can be not a constant, even if for the ellipsoids
or for the general hyperquadrics. This significant difference explains partially
why the proof of Theorem \ref{thm:1.2} is complicated and very different from
that of Blaschke and Deicke's theorem. To have a better understanding of these
respects, we would suggest the readers to compare the proof of Theorem 3.35, p.145
in \cite{LSZH} and that of Theorem 5 in \cite{L}.
\end{remark}


\numberwithin{equation}{section}
\section{Preliminaries}\label{sect:2}

In this section, we briefly recall some basic facts about centroaffine
hypersurfaces. More details are referred to the monographs \cite{LSZH,NS,SSV}
and the references \cite{LW,W}.

Let $\mathbb{R}^{n+1}$ be the $(n+1)$-dimensional affine space equipped with
its canonical flat connection $D$. Let $M^n$ be a connected $n$-dimensional
smooth manifold. An immersion $x: M^n\rightarrow \mathbb{R}^{n+1}$ is said to
be a centroaffine hypersurface if, for each point $x\in M^n$, the position
vector $x$ from the origin $O\in\mathbb{R}^{n+1}$ is transversal to the tangent
space $T_xM^n$ of $M^n$ at $x$. In that situation, the position vector $x$
defines the {\it centroaffine normalization} modulo orientation. For any vector
fields $X$ and $Y$ tangent to $M^n$, we have the centroaffine formula of Gauss:
\begin{equation}\label{eqn:2.1}
D_{X}x_{*}(Y)=x_{*}(\nabla_{X}Y) + h(X,Y)(-\varepsilon x) ,
\end{equation}
where $\varepsilon=1$ or $-1$. Moreover, associated with \eqref{eqn:2.1} we will
call $-\varepsilon x$, $\nabla$ and $h$ the centroaffine normal, the induced
(centroaffine) connection and the centroaffine metric, respectively. In this
paper, we will consider only locally strongly convex centroaffine hypersurfaces
such that the bilinear $2$-form $h$ defined by \eqref{eqn:2.1} is definite; and
we will choose $\varepsilon$ such that the centroaffine metric $h$ is positive
definite.

Let $x: M^n\rightarrow \mathbb{R}^{n+1}$ be a locally strongly convex centroaffine
hypersurface and $\hat\nabla$ be the Levi-Civita connection of its centroaffine
metric $h$. Then the tensor $K$, defined by $K(X,Y):=K_XY:=\nabla_XY-\hat{\nabla}_XY$,
is called the {\it difference tensor} of the centroaffine hypersurface. It is
symmetric as both connections $\nabla$ and $\hat\nabla$ are torsion free.
Let $\hat{R}$ denote the Riemannian curvature tensor of the centroaffine
metric $h$, then the following Gauss and Codazzi equations hold:
\begin{equation}\label{eqn:2.2}
\hat{R}(X,Y)Z=\varepsilon(h(Y,Z)X-h(X,Z)Y)-[K_{X}, K_{Y}]Z,
\end{equation}
\begin{equation}\label{eqn:2.3}
(\hat{\nabla}_ZK)(X,Y)=(\hat{\nabla}_XK)(Z,Y).
\end{equation}

Moreover, we further have the following totally symmetry equation
\begin{equation}\label{eqn:2.4}
h((\hat{\nabla}_ZK)(X,Y),W)=h((\hat{\nabla}_XK)(Z,W),Y).
\end{equation}

Associated to a centroaffine hypersurface $x: M^n\rightarrow \mathbb{R}^{n+1}$,
we can define the {\it Tchebychev form} $T^\sharp$ and the {\it Tchebychev vector
field} $T$ in implicit form by
\begin{equation}\label{eqn:2.5}
T^\sharp(X)=\tfrac1n{\rm Tr}\,(K_{X}),\ \ \ h(T, X)=T^\sharp(X),\ \ \forall\, X\in TM^n.
\end{equation}

Moreover, using the difference tensor $K$ and the Tchebychev vector field $T$,
one can further define the symmetric {\it traceless difference tensor} $\tilde{K}$ by
\begin{equation}\label{eqn:2.6}
\tilde{K}(X,Y):=K(X,Y)-\tfrac{n}{n+2}\big[h(X,Y)T+h(X,T)Y
+h(Y,T)X\big].
\end{equation}

It is well-known that $\tilde{K}$ vanishes if and only if $x(M^n)$
lies in a hyperquadric (cf. Section 7.1 in \cite{SSV}; Lemma 2.1 and
Remark 2.2 in \cite{LLSSW}).

As have been stated in the Introduction, the {\it centroaffine shape operator}
$\mathcal{T}$ of a centroaffine hypersurface $x:M^n\rightarrow \mathbb{R}^{n+1}$,
introduced by C.P. Wang \cite{W} and is also called the {\it Tchebychev operator},
is a homomorphism mapping $\mathcal{T}: TM\rightarrow TM$, defined by
\begin{equation}\label{eqn:2.7}
\mathcal{T}(X):=\hat{\nabla}_XT,\ \ \forall\,X\in TM^n.
\end{equation}
Then, the well-defined function $H:=\frac{1}{n}{\rm Tr}\,\mathcal{T}$ was named
as the {\it centroaffine mean curvature} of $x$. This is a meaningful terminology
because, according to C.P. Wang \cite{W}, the hypersurfaces with $H=0$ are exactly
the critical hypersurfaces of the centroaffine area functional. Moreover, related
to the centroaffine shape operator $\mathcal{T}$, it is interesting to consider an
important subclass of the centroaffine hypersurfaces, named as the {\it Tchebychev
hypersurfaces}, which is defined as below:

\begin{definition}[\cite{LW1}]\label{def:2.1}
Let $x:M^n\rightarrow \mathbb{R}^{n+1}$ be a centroaffine hypersurface such that
its Tchebychev operator $\mathcal{T}$ is proportional to the identity isomorphism
${\rm id}:TM^n\to TM^n$, i.e., $\mathcal{T}=\frac1n({{\rm div}\,T})\,{\rm id}$.
Then, $x$ is called a Tchebychev hypersurface.
\end{definition}

As it was pointed out in \cite{LSW} that the Tchebychev hypersurfaces satisfy
certain systems of second order PDE, and some of these systems play an important
role in the general context of conformal geometry. In this context, we shall
further emphasize the following important Riemannian geometric characterization
of the Tchebychev hypersurfaces:
\begin{lemma}\label{lem:2.1}
For a centroaffine hypersurface $x: M^n\rightarrow \mathbb{R}^{n+1}$, the Tchebychev
vector field $T$ is a closed vector field in the sense that the Tchebychev form $T^\sharp$
is a closed form. Moreover, a centroaffine hypersurface $x: M^n\rightarrow \mathbb{R}^{n+1}$
is a Tchebychev hypersurface if and only if, associated to the centroaffine metric $h$,
its Tchebychev vector field $T$ is a conformal vector field.
\end{lemma}
\begin{proof}
The first statement, which is equivalent to that the Tchebychev operator $\mathcal{T}$ is
self-adjoint with respect to the centroaffine metric $h$, was first shown by C.P. Wang \cite{W}.
Next, let $x: M^n\rightarrow \mathbb{R}^{n+1}$ be a centroaffine hypersurface with Tchebychev
vector field $T$. If it is a Tchebychev hypersurface, then we have
\begin{equation}\label{eqn:2.8}
(\mathcal{L}_Th)(X,Y)=h(\hat{\nabla}_XT,Y)+h(X,\hat{\nabla}_YT)=\tfrac2n({\rm div}\,T)h(X,Y),
\end{equation}
where $\mathcal{L}_T$ denotes the Lie derivative with respect to the Tchebychev vector field $T$.
This shows that $\mathcal{L}_Th=\tfrac2n({\rm div}\,T)h$. Thus, $T$ is a conformal vector field.
Conversely, assume that $T$ is a conformal vector field relative to $h$, i.e., it holds that
$$
(\mathcal{L}_Th)(X,Y)=2fh(X,Y)
$$
for any vector fields $X,Y$ and some smooth function $f$ on $M^n$. Then, by using that
$(\mathcal{L}_Th)(X,Y)=h(\hat{\nabla}_XT,Y)+h(X,\hat{\nabla}_YT)$ and the self-adjointness
of $\mathcal{T}$, namely that
\begin{equation}\label{eqn:2.9}
h(\hat{\nabla}_XT,Y)=h(\hat{\nabla}_YT,X),
\end{equation}
we derive $\hat{\nabla}_XT=fX$ for any vector field $X\in TM^n$. It follows that
$x: M^n\rightarrow \mathbb{R}^{n+1}$ is a Tchebychev hypersurface.
\end{proof}

Before concluding this section, we would emphasize that Riemannian manifolds with closed
conformal vector fields have been extensively studied, see the papers e.g. \cite{Ca,DRS,O,ST,TW,T}.
In next sections, we shall work for the application of the useful characterization of the
centroaffine Tchebychev hypersurfaces, established by Lemma \ref{lem:2.1}, so as to complete
the proof of Theorem \ref{thm:1.2}.

\section{Local properties of the Tchebychev hypersurfaces}\label{sect:3}

In this section, we will study the local properties of centroaffine Tchebychev
hypersurfaces in $\mathbb{R}^{n+1}$. Since our concern is the PROBLEM, and that
we already have Theorem \ref{thm:1.1}, in sequel we assume that $n\geq3$.

Recall that in \cite{LSW}, Liu, Simon and Wang established several local geometric
characterizations of the Tchebychev hypersurfaces and, as a corollary, they showed
that any quadric is a Tchebychev hypersurface.

In view of Lemma \ref{lem:2.1}, and according to Lemma 1 of \cite{RU} (cf. also
Lemma 1 of \cite{CMU}) which collects some results about Riemannian manifolds
admitting closed and conformal vector fields, we immediately obtain the following
lemma.

\begin{lemma}[cf. \cite{RU}]\label{lem:3.1}
Let $x: M^n\rightarrow\mathbb{R}^{n+1}$ be a centroaffine Tchebychev hypersurface
with nontrivial Tchebychev vector field $T$ and $\mathcal{T}=\alpha\,{\rm id}$.
Then we have:
\begin{enumerate}
\item[(i)] The norm $\|T\|$ with respect to the centroaffine metric $h$, the function
$\alpha$ and the curvature tensor $\hat{R}$ satisfy the following relations:
\begin{equation}\label{eqn:3.1}
\begin{aligned}
\hat{\nabla}\|T\|^2&=2\alpha T,\ \ \|T\|^2\hat{\nabla}\alpha=T(\alpha)T,\\
\|T\|^2\hat{R}(X,Y)T&=-T(\alpha)(h(T,Y)X-h(T,X)Y).
\end{aligned}
\end{equation}
\item[(ii)] The zeros of $T$ is a discrete set. Moreover, $T$ has nonvanishing divergence
at its zeros.
\item[(iii)] If we denote $\widetilde{M}=\{p\in M^n\mid T(p)\neq0\}$, then the distribution
$$
p\in \widetilde{M}\to \mathfrak{D}(p):=\{v\in T_pM^n\mid p\in\widetilde{M}\ {\rm and}\ h(v,T)=0\}
$$
defines an umbilical foliation on $(\widetilde{M},h)$. In particular, the functions
$\|T\|^2$ and $\alpha$ are constant on the connected leaves of $\mathfrak{D}$.
\item[(iv)] If $g=\|T\|^{-2}h$, then $(\widetilde{M},g)$ is locally isometric to
$(I\times N,dt^2\oplus g')$ and $T=(\partial/\partial t,0)$, where $I$
is an open interval in $\mathbb{R}$, $\{t\}\times N$ is a leaf of the
foliation $\mathfrak{D}$ for any $t\in\mathbb{R}$.
\end{enumerate}
\end{lemma}

Next, closely related to the proof of Theorem \ref{thm:1.2}, we first study
centroaffine Tchebychev hypersurfaces which satisfy the condition $K_TT=\lambda T$
for some function $\lambda\in C^\infty(M^n)$. Following the notations of
Lemma \ref{lem:3.1}, we begin with the following lemma.

\begin{lemma}\label{lem:3.2}
Let $x: M^n\rightarrow\mathbb{R}^{n+1}$ be a centroaffine Tchebychev hypersurface
whose Tchebychev vector field $T$ is nontrivial and satisfies $K_TT=\lambda T$. Then:
\begin{enumerate}
\item[(i)] For any point $p\in \widetilde{M}$, the eigenvalues
$\{\lambda_i\}_{2\leq i\leq n}$ of $K_T$ on $\mathfrak{D}(p)$ satisfy the quadratic
equation
\begin{equation}\label{eqn:3.2}
\lambda_i^2-\lambda\lambda_i+\varepsilon\|T\|^2+\alpha'=0,
\end{equation}
where $\alpha':=\frac{d\alpha}{dt}$. In particular, at most two of $\{\lambda_i\}_{2\leq i\leq n}$
are distinct.

\item[(ii)] If $V,W\in \mathfrak{D}$ are eigenvectors of $K_T$ corresponding to
different eigenvalues, then $K(V,W)=0$.

\item[(iii)] The eigenvalues of $K_T$ are constant on the connected leaves of
the foliation $\mathfrak{D}$.
\end{enumerate}
\end{lemma}
\begin{proof}
At any $p\in\widetilde{M}$, since $T$ is an eigenvector of $K_{T}$ and $K_{T}$ is self-adjoint
with respect to the centroaffine metric, $K_{T}$ can be diagonalized on
$\mathfrak{D}$. Let $\{X_i\}_{2\le i\le n}\subset\mathfrak{D}$ be the mutually orthogonal eigenvectors
of $K_{T}$ with corresponding eigenvalues $\{\lambda_i\}_{2\le i\le n}$, i.e.,
$$
K_TX_i=\lambda_iX_i,\ \ i=2,\ldots,n.
$$
Then, the third equation in \eqref{eqn:3.1} implies that
\begin{equation}\label{eqn:3.3}
\hat{R}(X_i,T)T=-\alpha'X_i,\ \ \hat{R}(X_i,X_j)T=0.
\end{equation}

On the other hand, by using the Gauss equation, we obtain
\begin{align}\label{eqn:3.4}
\begin{split}
&\hat{R}(X_i,T)T=(\lambda_i^2-\lambda\lambda_i+\varepsilon\|T\|^2)X_i,\\
&\hat{R}(X_i,X_j)T=(\lambda_i-\lambda_j)K(X_i,X_j).
\end{split}
\end{align}
From \eqref{eqn:3.3} and \eqref{eqn:3.4}, the assertions (i) and (ii) follows.

If we derivate $K_TT=\lambda T$ with respect to $X\in TM^n$, we obtain
\begin{equation}\label{eqn:3.5}
(\hat{\nabla}_XK)(T,T)+2\alpha K_TX=X(\lambda)T+\alpha\lambda X.
\end{equation}
It follows that for any vector field $Y\in TM^n$ there holds
\begin{equation}\label{eqn:3.6}
h((\hat{\nabla}_XK)(T,T),Y)+2\alpha h(K_TX,Y)=X(\lambda)h(T,Y)+\alpha\lambda h(X,Y).
\end{equation}

From \eqref{eqn:3.6} and noticing that both $h((\hat{\nabla}_\cdot K)(\cdot,\cdot),\cdot)$
and $h(K(\cdot,\cdot),\cdot)$ are totally symmetric, we get
$X(\lambda)h(T,Y)=Y(\lambda)h(T,X)$ for any $X,Y\in TM^n$. It follows that
\begin{equation}\label{eqn:3.7}
X(\lambda)T=h(T,X)\hat{\nabla}\lambda,\ \ X\in TM^n.
\end{equation}
Hence, we have $X(\lambda)=0$ for any $X\in\mathfrak{D}$.

From item (iii) of Lemma \ref{lem:3.1}, we know that $X(\alpha)=0$ for $X\in\mathfrak{D}$.
This, together with (iv) of Lemma \ref{lem:3.1}, implies that $X(\alpha')=0$ for any
$X\in\mathfrak{D}$. It follows that the solutions of \eqref{eqn:3.2} are constant on
each connected leaves of the foliation $\mathfrak{D}$. Thus the assertion (iii) follows.
\end{proof}

Now, we can further prove the following proposition.
\begin{proposition}\label{prop:3.1}
Let $x: M^n\rightarrow\mathbb{R}^{n+1}$ be a centroaffine Tchebychev hypersurface
with nontrivial $T$ such that $K_TT=\lambda T$. Then one of the following two cases occurs:
\begin{enumerate}
\item[(i)] $\widetilde{M}=M$, i.e. $T$ has no zeros; $\hat{\nabla}T=0$, and $K_{T}$ has
exactly two distinct constant eigenvalues on $\mathfrak{D}$;

\item[(ii)] $K_TV=\mu V$ for any $V\in \mathfrak{D}$. Moreover, we have
\begin{equation}\label{eqn:3.8}
\mu'=(\lambda-\mu)\alpha,\ \ n\|T\|^2=\lambda+(n-1)\mu,
\end{equation}
and that $c_0:=\varepsilon \|T\|^2-\mu^2+\alpha^2$ is a constant.
\end{enumerate}
\end{proposition}
\begin{proof}
As $M^n$ is connected and the zeros of $T$ is isolated, the subset $\widetilde{M}$ is also
connected. By Lemma \ref{lem:3.2}, $K_T$ on $\mathfrak{D}$ has at most two distinct
eigenvalues. We put
$$
M_0=\{p\in\widetilde{M}\,|\,K_T\ {\rm has\ two\ distinct\ eigenvalues\ on}\ \mathfrak{D}(p)\}.
$$

First of all, we assume that $M_0\not=\emptyset$. Obviously, $M_0$ is an open subset of $\widetilde{M}$.

\vskip 2mm

{\bf Claim 1}. {\it $M_0$ is a closed subset of $\widetilde{M}$}.

\vskip 2mm

Let $\{\lambda_i\}_{2\le i\le n}$ be the eigenvalues of $K_T$ on $\mathfrak{D}$ and assume that
$$
\mu_1:=\lambda_2=\cdots=\lambda_m<\lambda_{m+1}=\cdots=\lambda_n=:\mu_2.
$$

Since $\mu_1$ and $\mu_2$ are continuous functions on $M_0$, we see that $m$
is a constant on each connected component of $M_0$. So
$$
\mathfrak{D}_i=\{V\in \mathfrak{D}\mid K_TV=\mu_i V\},\ \ i=1,2,
$$
define two distributions on the connected components of $M_0$ and $\mathfrak{D}=\mathfrak{D}_1\oplus\mathfrak{D}_2$.

For $V\in\mathfrak{D}_1$ and $W\in\mathfrak{D}_2$, by direct calculations and using
Lemma \ref{lem:3.2}, we obtain
$$
\begin{aligned}
\hat{\nabla}_Wh(K_TV,V)=&h((\hat{\nabla}_WK)(T,V),V)+\alpha h(K_WV,V)+2h(K_TV,\hat{\nabla}_WV)\\
=&h((\hat{\nabla}_WK)(T,V),V)+2\mu_1h(V,\hat{\nabla}_WV),
\end{aligned}
$$
and
\begin{equation*}
\hat{\nabla}_W(\mu_1h(V,V))=2\mu_1h(V,\hat{\nabla}_WV).
\end{equation*}

Comparing the above equations and using $h(K_TV,V)=\mu_1h(V,V)$, we get
\begin{equation}\label{eqn:3.9}
h((\hat{\nabla}_WK)(T,V),V)=0,\ \ V\in\mathfrak{D}_1, W\in \mathfrak{D}_2.
\end{equation}

Similarly, for $V\in\mathfrak{D}_1$ and $W\in\mathfrak{D}_2$, taking the
derivative of $h(K_TV,W)=0$ with respect to $V\in\mathfrak{D}_1$ and using
Lemma \ref{lem:3.2}, we obtain
\begin{equation}\label{eqn:3.10}
h((\hat{\nabla}_VK)(T,V),W)=(\mu_1-\mu_2)h(\hat{\nabla}_VV,W),\ \ V\in \mathfrak{D}_1, W\in \mathfrak{D}_2.
\end{equation}

Comparing \eqref{eqn:3.9}, \eqref{eqn:3.10} and using \eqref{eqn:2.4}, we derive
\begin{equation}\label{eqn:3.11}
h(\hat{\nabla}_VV,W)=0,\ \ V\in \mathfrak{D}_1, W\in \mathfrak{D}_2.
\end{equation}

On the other hand, the fact $\mathcal{T}=\alpha\,{\rm id}$ implies that
\begin{equation}\label{eqn:3.12}
h(\hat{\nabla}_VV,T)=-h(\hat{\nabla}_VT,V)=-\alpha h(V,V),\ \ V\in \mathfrak{D}_1.
\end{equation}
Thus we get
\begin{equation}\label{eqn:3.13}
\hat{\nabla}_VV=(\hat{\nabla}_VV)^1-\alpha \|T\|^{-2}h(V,V)T,\ \ V\in \mathfrak{D}_1,
\end{equation}
where $(\hat{\nabla}_VV)^1$ denotes the component of $\hat{\nabla}_VV$ on $\mathfrak{D}_1$.

Again, for $V\in\mathfrak{D}_1$ and $W\in\mathfrak{D}_2$, taking the covariant derivative
of $K(V,W)=0$ with respect to $V$, we obtain
\begin{equation*}
(\hat{\nabla}_VK)(V,W)+K(\hat{\nabla}_VV,W)+K({\nabla}_VW,V)=0.
\end{equation*}
Hence, we have
\begin{equation}\label{eqn:3.14}
h((\hat{\nabla}_VK)(V,W),W)+h(K(\hat{\nabla}_VV,K_WW)=0,\ \ V\in \mathfrak{D}_1, W\in \mathfrak{D}_2.
\end{equation}

Inserting \eqref{eqn:3.13} into \eqref{eqn:3.14}, we obtain
\begin{equation*}
h((\hat{\nabla}_VK)(V,W),W)=\alpha \|T\|^{-2}\mu_2h(V,V)h(W,W),\ \ V\in \mathfrak{D}_1, W\in \mathfrak{D}_2.
\end{equation*}
Then, the fact $h((\hat{\nabla}_WK)(W,V),V)=h((\hat{\nabla}_VK)(V,W),W)$ implies that
\begin{equation*}
h((\hat{\nabla}_VK)(V,W),W)=\alpha \|T\|^{-2}\mu_1h(V,V)h(W,W),\ \ V\in \mathfrak{D}_1, W\in \mathfrak{D}_2.
\end{equation*}

Comparing the above computations, we obtain,
\begin{equation*}
(\mu_1-\mu_2)\alpha \|T\|^{-2}h(V,V)h(W,W)=0,\ \ V\in \mathfrak{D}_1, W\in \mathfrak{D}_2.
\end{equation*}
It follows that $\alpha=0$ on $M_0$. Hence, we have
\begin{equation}\label{eqn:3.15}
\hat{\nabla}T=0,\ {\rm on}\ M_0.
\end{equation}

Taking $X=V\in\mathfrak{D}_1$ and $Y=W\in\mathfrak{D}_2$ in \eqref{eqn:3.6}, we get
\begin{equation}\label{eqn:3.16}
h((\hat{\nabla}_VK)(T,T),W)=0,\ \ V\in \mathfrak{D}_1, W\in \mathfrak{D}_2.
\end{equation}

For $V\in\mathfrak{D}_1$ and $W\in \mathfrak{D}_2$, taking the derivative of
$h(K_TV,W)=0$ with respect to $T$, we easily obtain
\begin{equation}\label{eqn:3.17}
h((\hat{\nabla}_VK)(T,T),W)+\mu_2h(\hat{\nabla}_TV,W)+\mu_1h(\hat{\nabla}_TW,V)=0.
\end{equation}

From \eqref{eqn:3.16} and \eqref{eqn:3.17}, we get
\begin{equation}\label{eqn:3.18}
(\mu_2-\mu_1)h(\hat{\nabla}_TV,W)=0,\ \ V\in \mathfrak{D}_1, W\in \mathfrak{D}_2.
\end{equation}

From \eqref{eqn:3.18} and noting that $h(\hat{\nabla}_TV,T)=-\alpha h(T,V)=0$,
we get
\begin{equation}\label{eqn:3.19}
\hat{\nabla}_TV\in \mathfrak{D}_1,\ \ V\in \mathfrak{D}_1.
\end{equation}

Then, for $V\in \mathfrak{D}_1$, taking the covariant derivative of $K_TV=\mu_1V$
with respect to $T$ and using \eqref{eqn:3.15}, we obtain
\begin{equation}\label{eqn:3.20}
(\hat{\nabla}_TK)(T,V)+K(\hat{\nabla}_TV,T)=T(\mu_1)V+\mu_1\hat{\nabla}_TV.
\end{equation}
Then, by \eqref{eqn:3.19} and \eqref{eqn:3.20}, we get
\begin{equation}\label{eqn:3.21}
h((\hat{\nabla}_TK)(T,V),V)=T(\mu_1)h(V,V).
\end{equation}

On the other hand, as $\alpha=0$ on $M_0$, from \eqref{eqn:3.5} and \eqref{eqn:3.15},
we have
\begin{equation}\label{eqn:3.22}
h((\hat{\nabla}_TK)(T,V),V)=0.
\end{equation}
Then \eqref{eqn:3.21} and \eqref{eqn:3.22} imply that
\begin{equation}\label{eqn:3.23}
T(\mu_1)=0\ {\rm on}\ M_0.
\end{equation}
This and (iii) of Lemma \ref{lem:3.2} show that $\mu_1$ is a constant on the
component of $M_0$.

Similarly, we can prove that $\mu_2$ is constant on the component of $M_0$.
So, by continuity of $\mu_1$ and $\mu_2$, $K_T$ restricted on $\mathfrak{D}$
has two distinct eigenvalues on the closure of $M_0$. Thus, $M_0$ is a
closed subset of $\widetilde{M}$ and Claim 1 is verified.

In summary, we have proved that there are only two possibilities:
Either $M_0=\widetilde{M}$ or $M_0=\emptyset$.

If $\widetilde{M}=M_0$, then by continuity $\hat{\nabla}T=0$ and $h(T,T)$ is
constant on $M$. Since by assumption $T\neq0$, we have proved that $T$ has no
zeros and $\widetilde{M}=M$. Thus, case (i) in Proposition \ref{prop:3.1} occurs.

If $M_0=\emptyset$, $K_T$ restricted on $\mathfrak{D}$ has only one eigenvalue,
denoted by $\mu$, which is constant on the leaves of $\mathfrak{D}$.
Hence $K_TV=\mu V$ for any $V\in\mathfrak{D}$.

Taking the derivative of $K_TV=\mu V$ with respect to $T$ and noting that
$\hat{\nabla}_TV\in\mathfrak{D}$, we obtain
\begin{equation}\label{eqn:3.24}
(\hat{\nabla}_TK)(T,V)+\alpha\mu V=\mu' V,\ \ V\in\mathfrak{D}.
\end{equation}

On the other hand, from \eqref{eqn:3.5}, we obtain
\begin{equation}\label{eqn:3.25}
(\hat{\nabla}_TK)(T,V)=(\lambda-2\mu)\alpha V,\ \ V\in\mathfrak{D}.
\end{equation}

From \eqref{eqn:3.24} and \eqref{eqn:3.25}, we get the first equation in \eqref{eqn:3.8}.

For any $p\in\widetilde{M}$, let $\{e_1=\|T\|^{-1}T,e_2,\ldots,e_n\}$ be an orthonormal
basis of $T_pM^n$. Then, by definition \eqref{eqn:2.5} we can derive that $T=\tfrac1n\sum_{i=1}^nK_{e_i}e_i$.
It follows that
\begin{equation}\label{eqn:3.26}
T=\frac{1}{n}\Big(\|T\|^{-2}K_TT+\sum_{i=2}^nK_{e_i}e_i\Big).
\end{equation}

Taking the inner product of both sides of \eqref{eqn:3.26} with $T$,
we get immediately the second equation in \eqref{eqn:3.8}.

By using \eqref{eqn:3.2}, the first two equations in \eqref{eqn:3.1} and
the first equation in \eqref{eqn:3.8}, direct calculations show that
$$
\hat{\nabla}(\varepsilon\|T\|^2-\mu^2+\alpha^2)=2\|T\|^{-2}\alpha(\mu^2-\lambda\mu
+\varepsilon\|T\|^2+\alpha')T=0.
$$
It follows that $\varepsilon\|T\|^2-\mu^2+\alpha^2=:c_0$ is a constant on $\widetilde{M}$.

We have proved that if $M_0=\emptyset$ then case (ii) in Proposition \ref{prop:3.1} occurs.
\end{proof}

As a crucial step to complete the proof of Theorem \ref{thm:1.2}, we intend to derive
a locally expression for centroaffine Tchebychev hypersurfaces which are assumed to
satisfy property (ii) of Proposition \ref{prop:3.1}. To achieve the purpose, we first
state the following lemma, whose proof is an easy computation.

\begin{lemma}\label{lem:3.3}
Assume that $\lambda(t), \alpha(t)$ and $\mu(t)$ are real-valued functions satisfying
$$
\alpha'=-\mu^2+\lambda \mu-\varepsilon\|T\|^2,\ \ \mu'=(\lambda-\mu)\alpha.
$$
Then the ordinary differential equation
$$
x_{tt}=(\lambda+\alpha) x_t-\varepsilon|T|^2x
$$
has two linear independent solutions that can be written as follows:
$$
x_1=e^{\int (\alpha+\mu) dt},\ \ x_2=x_1\int e^{\int(\lambda-2\mu-\alpha dt)}dt.
$$
\end{lemma}

Finally, as one main result of this section, we can prove the following

\begin{proposition}\label{prop:3.2}
Let $x: M^n\rightarrow\mathbb{R}^{n+1}$ be a centroaffine Tchebychev hypersurface
with $\mathcal{T}=\alpha\,{\rm id}$ such that $K_TT=\lambda T$ and case (ii) in
Proposition \ref{prop:3.1} occurs. Then $x$ can be written as
\begin{equation}\label{eqn:3.27}
x=\gamma_1(t)\varphi+\gamma_2(t)C,
\end{equation}
where $\varphi:N\rightarrow\mathbb{R}^{n}$ is an affine hypersphere, $C$ is a
nonzero constant vector in $\mathbb{R}^{n+1}$, and
\begin{equation}\label{eqn:3.28}
\gamma_1(t)=e^{\int (\alpha+\mu) dt},\ \gamma_2(t)=e^{\int(\alpha+\mu)dt}\int
e^{\int(\lambda-2\mu-\alpha) dt}dt.
\end{equation}

Moreover, the difference tensor $K$ of $x: M^n\rightarrow\mathbb{R}^{n+1}$ and the
difference tensor $K'$ of $\varphi:N\rightarrow\mathbb{R}^{n}$ are related by
\begin{equation}\label{eqn:3.29}
K(X,Y)=\|T\|^{-2}h(X,Y)\mu T+K'(X,Y),\ X,Y\in\mathfrak{D}.
\end{equation}
\end{proposition}
\begin{proof}
From item (iv) of Lemma \ref{lem:3.1}, the Riemannian manifold $(M^n,h)$ is locally
isometric to $(I\times N,\|T\|^2(dt^2+g'))$, where $N$ is the integral manifold of
$\mathfrak{D}$ and $g'$ is a metric defined on $N$. Therefore, we can take a local
coordinate $(u_1:=t,u_2,\ldots,u_n)$ of $M^n=I\times N$ so that the metric $h$ has
the following expression
$$
h=\|T\|^2\Big(dt^2+\sum_{i=2}^ng'_{ij}du_idu_j\Big).
$$

By using \eqref{eqn:2.1}, and case (ii) of Proposition \ref{prop:3.1}, we have
\begin{equation}\label{eqn:3.30}
x_{tt}=(\lambda+\alpha)x_t-\varepsilon\|T\|^2x,
\end{equation}
\begin{equation}\label{eqn:3.31}
x_{tu_i}=(\mu+\alpha)x_{u_i},\ \ i\geq 2.
\end{equation}

In our case, by \eqref{eqn:3.2}, \eqref{eqn:3.8} and Lemma \ref{lem:3.3},
we can solve \eqref{eqn:3.30} to obtain
\begin{equation}\label{eqn:3.32}
x=\gamma_1(t)\varphi(u_2,\ldots,u_n)+\gamma_2(t)\psi(u_2,\ldots,u_n).
\end{equation}
where $\varphi(u_2,\ldots,u_n)$ and $\psi(u_2,\ldots,u_n)$ are
$\mathbb{R}^{n+1}$-valued functions, $\gamma_1(t)$ and $\gamma_2(t)$ are
described by \eqref{eqn:3.28}.

Then, substituting \eqref{eqn:3.32} into \eqref{eqn:3.31}, we further derive
$\tfrac{\partial \psi}{\partial u_k}=0$ for $2\le k\le n$. This implies that
$\psi(u_2,\ldots,u_n)=:C$ is a constant vector in $\mathbb{R}^{n+1}$. Due to
that $x: M^n\rightarrow\mathbb{R}^{n+1}$ is a nondegenerate centroaffine
hypersurfaces, $C$ must be nonzero (if otherwise, we have $x_t=\gamma_1^{-1}\gamma_1'x$,
contradicting to that $x$ is a transversal vector). Now, we have
\begin{equation}\label{eqn:3.33}
x=\gamma_1(t)\varphi(u_2,\ldots,u_n)+\gamma_2(t)C.
\end{equation}
It follows that
\begin{equation}\label{eqn:3.34}
\varphi_{u_i}=\gamma_1(t)^{-1}x_{u_i},\ \ i\geq2.
\end{equation}
Thus, $\varphi$ defines an immersion from $N$ into $\mathbb{R}^{n+1}$.

\vskip 1mm

{\bf Claim 2}. $\varphi:N\to\mathbb{R}^{n+1}$ defines a locally strongly convex affine
hypersphere in an $n$-dimensional vector subspace $\mathbb{R}^{n}$ of $\mathbb{R}^{n+1}$.

\vskip 1mm

To verify Claim 2, noticing that $\|T\|^{-2}h(\partial u_i,\partial u_j)=g'(\partial u_i,\partial u_j)$
for $i,j\geq2$, and from \eqref{eqn:3.28} we can derive $x_t=\gamma_1(\alpha+\mu)\varphi+\gamma_2(\alpha+\mu)C
+\gamma_1e^{\int(\lambda-2\mu-\alpha) dt}C$. Then straightforward calculations by using \eqref{eqn:3.33}
and \eqref{eqn:3.34} give that
\begin{align}\label{eqn:3.35}
\begin{split}
\varphi_{u_iu_j}&=\gamma_1(t)^{-1}x_{u_iu_j}=\gamma_1(t)^{-1}(x_*(\nabla_{\partial u_i}\partial u_j)
-\varepsilon h(\partial u_i,\partial u_j)x)\\
&=\gamma_1(t)^{-1}\Big(x_*(\nabla^T_{\partial u_i}\partial u_j)
 +h(\hat{\nabla}_{\partial u_i}\partial u_j,\|T\|^{-1}T)\,\|T\|^{-1}x_t\\
&\qquad\qquad\ \ +h(K_{\partial u_i}\partial u_j,\|T\|^{-1}T)\,\|T\|^{-1}x_t-\varepsilon h(\partial u_i,\partial u_j)x\Big)\\
&=\gamma_1(t)^{-1}\Big(x_*(\nabla^T_{\partial u_i}\partial u_j)\Big)
+\Big[c_0\varphi+\Big(c_0\int
e^{\int(\lambda-2\mu-\alpha) dt}dt\\
&\qquad\qquad\ \ \ +(\mu-\alpha)e^{\int(\lambda-2\mu-\alpha) dt}\Big)C\Big]\|T\|^{-2}h(\partial u_i,\partial u_j)\\
&=\varphi_*(\nabla^T_{\partial u_i}\partial u_j)
+\Big[c_0\varphi+\Big(c_0\int e^{\int(\lambda-2\mu-\alpha) dt}dt\\
&\qquad\qquad\ \ \ +(\mu-\alpha)e^{\int(\lambda-2\mu-\alpha) dt}\Big)C\Big]g'(\partial u_i,\partial u_j),\ \ i,j\geq2,
\end{split}
\end{align}
where $c_0:=\varepsilon \|T\|^2-\mu^2+\alpha^2$ is a constant as described in
Proposition \ref{prop:3.1}, and $\nabla^T_{\partial u_i}\partial u_j$ denotes
the tangent component of $\nabla_{\partial u_i}\partial u_j$ in $TN$ .

\vskip 1mm

Now, we consider two possibilities:

\vskip 1mm

\textbf{Case I.} $c_0=0$, i.e. $\varepsilon \|T\|^2-\mu^2+\alpha^2=0$.

By direct calculations, we can show that $(\mu-\alpha)e^{\int(\lambda-2\mu-\alpha)dt}$
is a constant. If $(\mu-\alpha)e^{\int(\lambda-2\mu-\alpha) dt}=0$, then
$\mu-\alpha=0$ which, together with $c_0=0$, implies that $T=0$.
This contradict to the assumption $T\neq0$. Hence $(\mu-\alpha)e^{\int(\lambda-2\mu-\alpha)dt}\neq0$.
Then, denoting the non-zero constant vector $(\mu-\alpha)e^{\int(\lambda-2\mu-\alpha) dt}C$ still by $C$,
we get the expression
\begin{equation}\label{eqn:3.36}
\varphi_{u_iu_j}=\varphi_*(\nabla^T_{\partial u_i}\partial u_j)
+g'(\partial u_i,\partial u_j)C,\ \ i,j\geq2.
\end{equation}

\textbf{Case II.} $c_0\neq0$. In this case, we have
\begin{align}\label{eqn:3.37}
\begin{split}
\varphi_{u_iu_j}=&\gamma_1(t)^{-1}\Big(x_*(\nabla^T_{\partial u_i}\partial u_j)\Big)
+c_0\Big[\varphi+\Big(\int e^{\int(\lambda-2\mu-\alpha) dt}dt\\
&\ \ \ +c_0^{-1}(\mu-\alpha)e^{\int(\lambda-2\mu-\alpha) dt}\Big)C\Big]g'(\partial u_i,\partial u_j),\ \ i,j\geq2.
\end{split}
\end{align}

Moreover, direct calculations show that the term
$$
\int e^{\int(\lambda-2\mu-\alpha) dt}dt+c_0^{-1}(\mu-\alpha)e^{\int(\lambda-2\mu-\alpha)dt}
$$
is a constant. Then, denoting $(\int e^{\int(\lambda-2\mu-\alpha) dt}dt+c_0^{-1}(\mu-\alpha)e^{\int(\lambda-2\mu-\alpha)dt})C$
still by $C$, we obtain
\begin{align}\label{eqn:3.38}
\begin{split}
\varphi_{u_iu_j}=&\varphi_*(\nabla^T_{\partial u_i}\partial u_j)
+c_0(\varphi+C)g'(\partial u_i,\partial u_j),\ \ i,j\geq2.
\end{split}
\end{align}

Therefore, for both cases, \eqref{eqn:3.36} (resp. \eqref{eqn:3.38}) implies that in Case I
(resp. Case II) the image of $\varphi$ is contained in an $n$-dimensional linear subspace
$\mathbb{R}^{n}$ of $\mathbb{R}^{n+1}$, and the immersion $\varphi:N\rightarrow\mathbb{R}^{n}$
can be interpreted as a relative hypersphere with respect to the relative normal
vector field $C$ (resp. $c_0(\varphi+C)$), with induced connection $\nabla^T$ and
relative metric $g'$ (which by definition is definite), respectively.

Denote by $\hat\nabla'$ the Levi-Civita connection of $g'$. Using $h=\|T\|^2(dt^2+g')$
and the Koszul's formula, we have the calculation
$$
\begin{aligned}
\|T\|^{-2}h(\hat\nabla'_{\partial u_i}\partial u_j,\partial u_k)
&=g'(\hat\nabla'_{\partial u_i}\partial u_j,\partial u_k)\\
&=\|T\|^{-2}h(\hat{\nabla}_{\partial u_i}\partial u_j,\partial u_k),\ \ i,j\geq2.
\end{aligned}
$$
This shows that $\hat\nabla'_{\partial u_i}\partial u_j=\hat{\nabla}^T_{\partial u_i}\partial u_j$
for $i,j\geq2$, and therefore, the difference tensor $K'$ of $\varphi:N\rightarrow\mathbb{R}^{n}$
is given by
\begin{equation}\label{eqn:3.39}
K'_{\partial u_i}\partial u_j=K^T_{\partial u_i}\partial u_j,\ \ i,j\geq2,
\end{equation}
where $\hat{\nabla}^T_{\partial u_i}\partial u_j$ and $K^T_{\partial u_i}\partial u_j$
denote the tangent parts of $\hat{\nabla}_{\partial u_i}\partial u_j$ and
$K_{\partial u_i}\partial u_j$ in $TN$, respectively.

Let $((g')^{ij})$ (resp. $(h^{AB})$) denote the inverse matrix of $(g'_{ij})$
(resp. $(h_{AB})$), and $h_{AB}=h(\partial u_A,\partial u_B)$ for $A,B\ge1$.
Then, by using \eqref{eqn:3.39}, for any $X\in TN$, we can get the calculation that
\begin{align}\label{eqn:3.40}
\begin{split}
{\rm Tr}\, K'_X&=\sum_{i,j=2}^n(g')^{ij}g'(K'_{X}\partial u_i,\partial u_j)\\
&=\sum_{i,j=2}^nh^{ij}h(K_{X}\partial u_i,\partial u_j)\\
&=nh(T,X)-h^{11}h(K_{X}T,T)\\
&=0.
\end{split}
\end{align}
This verifies Claim 2 that $\varphi:N\rightarrow \mathbb{R}^{n}$ is actually an affine hypersphere.

Finally, by \eqref{eqn:3.39} and the definition of the difference tensor, \eqref{eqn:3.29}
immediately follows.
\end{proof}

\section{The Completion of Theorem \ref{thm:1.2}'s Proof}\label{sect:4}
We first show that for a centroaffine Tchebychev hyperovaloid, the Tchebychev
vector field $T$ and the difference tensor field $K$ satisfy an important
relation.
\begin{lemma}\label{lm:4.1}
Let $x: M^n\rightarrow\mathbb{R}^{n+1}$ be a centroaffine Tchebychev
hyperovaloid. Then its Tchebychev vector field $T$ and difference tensor
field $K$ satisfy the relation
\begin{equation}\label{eqn:4.1}
K_TT=\tfrac{3n}{n+2}\|T\|^2T.
\end{equation}
\end{lemma}
\begin{proof}
Put $Z:=K_TT-\frac{3n}{n+2}\|T\|^2T$. By a direct calculation, we obtain
\begin{equation}\label{eqn:4.2}
\begin{aligned}
h(\hat{\nabla}_XZ,Y)=&h((\hat{\nabla}_XK)(T,T),Y)+\tfrac2n({\rm div}\,T)h(K_XY,T)\\&-
\tfrac3{n+2}({\rm div}\,T)\big(2h(X,T)h(Y,T)+h(T,T)h(X,Y)\big).
\end{aligned}
\end{equation}
From \eqref{eqn:4.2}, it is easily seen that $Z$ is a closed vector field which
satisfies ${\rm div} Z=0$. So $Z$ is a harmonic vector field on $M^n$. Notice
that $M^n$ is diffeomorphic to a sphere, whereas on the sphere there are no
nontrivial harmonic vector fields. Hence, $Z$ vanishes identically, and we get
\eqref{eqn:4.1}.
\end{proof}

It is well known that the centroaffine Tchebychev form $T^\sharp$ can be expressed by the
equiaffine support function $\rho$ (cf. \cite{LSW,STV}):
\begin{equation}\label{eqn:4.3}
T^\sharp=\tfrac{n+2}{2n}d\ln\rho.
\end{equation}

Put $f:=\frac{n+2}{2n}\ln\rho$. Then by \eqref{eqn:4.3}, we can write $T=\hat{\nabla}f$.
It follows that
\begin{equation}\label{eqn:4.4}
{\rm Hess} f(X,Y)=\tfrac1n({\rm div}\,T)h(X,Y).
\end{equation}

If $T\equiv0$, then, as a centroaffine hypersurface, $x$ is an affine hypersphere
centered at the origin $O\in\mathbb{R}^{n+1}$. By the theorem of Blaschke and Deicke,
$x: M^n\rightarrow\mathbb{R}^{n+1}$ is an ellipsoid centered at the origin of
$\mathbb{R}^{n+1}$.

Next, we assume that $T\neq 0$. Then, since $M^n$ is compact and is diffeomorphic
to a sphere, according to \cite{T} and \eqref{eqn:4.4} we know that $(M^n,h)$ is
conformally equivalent to a round sphere and the number of isolated zeros of $T$
is $2$ (cf. also Theorem 4.6 of \cite{LSW}).
On the other hand, from Lemma \ref{lm:4.1} and Proposition \ref{prop:3.1}, we
see that, in order to complete the proof of Theorem \ref{thm:1.2}, we are left
to study case (ii) in Proposition \ref{prop:3.1}. Then, by using Proposition
\ref{prop:3.2}, we know that $x:M^n\to\mathbb{R}^{n+1}$ reduces to be
\begin{equation*}
x=\gamma_1(t)\varphi+\gamma_2(t)C.
\end{equation*}
where $\varphi:N\rightarrow\mathbb{R}^{n}$ is an affine hypersphere, $C$ is a
nonzero constant vector in $\mathbb{R}^{n+1}$, $\gamma_1(t)$ and $\gamma_2(t)$ are
described by \eqref{eqn:3.28}.

Since, according to pp.18-19 of \cite{Ch}, the umbilicity of submanifolds is invariant under
a conformal transformation of the ambient Riemannian manifold, and in a round sphere the
umbilical hypersurfaces are spheres, we obtain that the leaves $N$ of the umbilical foliation
$\mathfrak{D}$ are spheres. This implies that $\varphi:N\rightarrow\mathbb{R}^{n}$ is a
locally strongly convex affine hypersphere which is compact and without boundary. Then, by
the theorem of Blaschke and Deicke $\varphi:N\rightarrow\mathbb{R}^{n}$
is an ellipsoid, and thus $K'=0$. Hence, from the second equation in \eqref{eqn:3.8}, \eqref{eqn:3.29}
and \eqref{eqn:4.1}, we obtain that
\begin{equation}\label{eqn:4.5}
\left\{
\begin{aligned}
&K_TT=\tfrac{3n}{n+2}\|T\|^2T,\\
&K_TV=\tfrac{n}{n+2}\|T\|^2 V,\ \ V\in \mathfrak{D},\\
&K_VW=\tfrac{n}{n+2}h(V,W)T,\ \ V,W\in \mathfrak{D}.
\end{aligned}
\right.
\end{equation}
Then, by direct calculations, we can show that $\tilde{K}=0$ on $\widetilde{M}$.
By continuity, $\tilde{K}=0$ holds on the whole $M^n$. It follows that
$x: M^n\rightarrow\mathbb{R}^{n+1}$ is an ellipsoid and, as a centroaffine
hypersurface, the origin of $\mathbb{R}^{n+1}$ must be in the inside of $x(M^n)$.

This completes the proof of Theorem \ref{thm:1.2}.\qed


\vskip 3mm
\begin{flushleft}
Xiuxiu Cheng and Zejun Hu:\\
{\sc School of Mathematics and Statistics, Zhengzhou University,\\
Zhengzhou 450001, People's Republic of China.}\\
E-mails: chengxiuxiu1988@163.com; huzj@zzu.edu.cn.

\vskip 2mm
Luc Vrancken:

{\sc Universit\'e Polytechnique Hauts de France, F-59313 Valenciennes, FRANCE;\ \
KU Leuven, Department of Mathematics, Celestijnenlaan 200B--Box 2400, BE-3001 Leuven, Belgium.}\\
E-mail: luc.vrancken@univ-valenciennes.fr.

\end{flushleft}

\end{document}